\documentclass{article}
\usepackage[utf8]{inputenc}
\usepackage{float}
\usepackage{ amsmath }
\usepackage{amsthm}
\usepackage{ dsfont }
\usepackage{algpseudocode}
\usepackage{amssymb}
\usepackage{diagbox}
\usepackage{graphicx}
\usepackage{slashbox}
\usepackage{ wasysym }
\usepackage{hyperref}
\usepackage{tikz}
\usepackage{ mathrsfs }
\usepackage{tikz}
\usepackage{biblatex}
\usetikzlibrary{matrix,shapes,arrows}

\newtheorem{theorem}{Theorem}

\newtheorem{lemma}{Lemma}

\newtheorem{remark}{Remark}

\title{ Infinite Primes From Integer Partitions}

\begin{document}

\author{Anton Shakov}

\maketitle

\begin{abstract}
\href{https://mathworld.wolfram.com/FerrersDiagram.html}{Ferrers diagrams} \cite{weisstein2003ferrers} are used to visually represent integer partitions. We describe a way to use Ferrers diagrams to uniquely represent integers in terms of their prime factors. This leads to a lower bound on the number of primes less than a given integer, namely $\pi(x) \geq \frac{\lfloor \lg x \rfloor}{\lg (\lfloor \lg x \rfloor + 1)}$ where $\pi(x)$ is the prime counting function and $\lg(x)$ denotes the base 2 logarithm. This results in a new proof of the infinitude of primes.
\end{abstract}

\section{Introduction}

First, we give examples of  a few Ferrers diagrams corresponding to integer partitions. The integer $3$ can be partitioned in three ways as $1 + 1 + 1, \ 1 + 2,$ and $3$. The Ferrers diagrams corresponding to these partitions will be 

$$
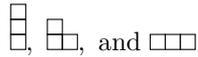
\begin{tikzpicture}[scale=0.2]
   \draw (0,0) rectangle +(1,1);
   \draw (0,1) rectangle +(1,1);
   \draw (0,2) rectangle +(1,1);
\end{tikzpicture}, \
\begin{tikzpicture}[scale=0.2]
   \draw (0,0) rectangle +(1,1);
   \draw (1,0) rectangle +(1,1);
   \draw (0,1) rectangle +(1,1);
\end{tikzpicture}, \text{ and }
\begin{tikzpicture}[scale=0.2]
   \draw (0,0) rectangle +(1,1);
   \draw (1,0) rectangle +(1,1);
   \draw (2,0) rectangle +(1,1);
\end{tikzpicture}
$$

respectively. We use the Fundamental Theorem of Arithmetic to construct a bijection $\mathcal{F}$ from the natural numbers $\mathbf{N} = \{1, 2, 3, \cdots \}$ to Ferrers diagrams. For brevity, we often refer to Ferrers diagrams simply as figures. We then use $\mathcal{F}$ to define a partial ordering on the set of Ferrers diagrams $\mathcal{F}(\mathbf{N})$ from which we are able to easily extract the estimate $\pi(x) \geq \frac{\lfloor \lg x \rfloor}{\lg (\lfloor \lg x \rfloor + 1)}$. Since the function $\frac{\lfloor \lg x \rfloor}{\lg (\lfloor \lg x \rfloor + 1)}$ is unbounded as $x \to \infty$, it follows that the number of primes must be infinite.

\section{The bijection $\mathcal{F}(n)$}

Let $n \in \mathbf{N}$. We describe how to draw the figure $\mathcal{F}(n)$. 

The Fundamental Theorem of Arithmetic states that $n$ has a unique decomposition into prime numbers $p_\alpha$ up to the order in which we multiply.

\begin{equation}\label{decomp}
 n = p_{\alpha_1}^{\beta_1} p_{\alpha_2}^{\beta_2} \cdots p_{\alpha_s}^{\beta_s}
\end{equation}

The $\alpha_i$ denotes the index of each prime as they appear in the set $\mathbf{N}$. We have $p_1 = 2$, $p_2 = 3$, $p_3 = 5$, etc. To avoid ambiguity in drawing the Ferrers diagram $\mathcal{F}(n)$, assume $\alpha_i < \alpha_{i+1}$.

We construct the corresponding Ferrers diagram. For each prime $p_{\alpha_i}$ we stack a row of $\alpha_i$ squares for each instance that $p_{\alpha_i}$ divides $n$. Equation~\ref{decomp} instructs us to stack $\beta_i$ rows of length $\alpha_i$ for $i \leq s$.

Therefore the general figure $\mathcal{F}(n)$ is drawn as

$$
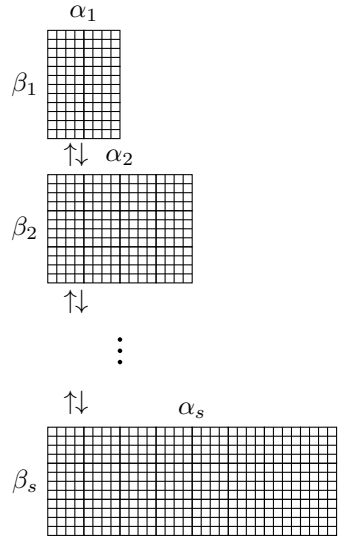
\begin{tikzpicture}[scale=0.15]

    \def\squareSize{3.2} 

    \def\smallSquareSize{0.8}

    \def\vsep{2.5}
    \def\firstVsep{1} 

    \def\initY{0}

    \def\alphas{2,4,8}

    \foreach \width [count=\i] in \alphas {

        \ifnum\i=2
            \pgfmathsetmacro\startY{\initY - (\i-1)*\squareSize*\firstVsep - (\i-1)*\squareSize*3}
        \else
            \pgfmathsetmacro\startY{\initY - (\i-1)*\squareSize*\vsep - (\i-1)*\squareSize*3}
        \fi

        \foreach \x in {1,...,\width} {
            \foreach \y in {1,2,3} {
                \draw (\x*\squareSize-\squareSize, \startY-\y*\squareSize+\squareSize) rectangle (\x*\squareSize, \startY-\y*\squareSize);

                \foreach \a in {1,2,3} {
                    \draw (\x*\squareSize-\squareSize, \startY-\y*\squareSize+\a*\smallSquareSize) -- (\x*\squareSize, \startY-\y*\squareSize+\a*\smallSquareSize);
                    \draw (\x*\squareSize-\squareSize+\a*\smallSquareSize, \startY-\y*\squareSize) -- (\x*\squareSize-\squareSize+\a*\smallSquareSize, \startY-\y*\squareSize+\squareSize);
                }
            }
        }


        \ifnum\i=3
            \node[above] at (\width*\squareSize/2, \startY) {$\alpha_s$};
            \node[left] at (0, \startY-1.5*\squareSize) {$\beta_s$};
        \else
            \node[above] at (\width*\squareSize/2, \startY) {$\alpha_{\i}$};
            \node[left] at (0, \startY-1.5*\squareSize) {$\beta_{\i}$};
        \fi
    }

    \pgfmathsetmacro\dotY{\initY - 2*3*\squareSize - \firstVsep - \vsep + 1.5*\squareSize}

    \foreach \i in {1,2,3}
    {
        \fill (2*\squareSize, 42*\dotY - \i*1cm) circle (0.2);
    }

\node at (2,-11) {$\uparrow$};
\node at (3,-11) {$\downarrow$};
\node at (2,-24) {$\uparrow$};
\node at (3,-24) {$\downarrow$};
\node at (2,-33) {$\uparrow$};
\node at (3,-33) {$\downarrow$};

\end{tikzpicture}$$

\

The set of all Ferrers diagrams can therefore be written as $\mathcal{F}(\mathbf{N})$. The fact that $\mathcal{F}$ is a bijection follows immediately from the uniqueness of a number's decomposition into primes. We draw the figures $\mathcal{F}(n)$ for $n \leq 10$.

\begin{table}[H]
\renewcommand{\arraystretch}{2.15}
\centering
\begin{tabular}{|c|c|>{\centering\arraybackslash}p{3cm}|}
\hline
$n$ & Prime factors of $n$ & $\mathcal{F}(n)$ \\
\hline
1 & - & $\emptyset$  \\
2 & $p_{_1} $  & \begin{tikzpicture}[scale=0.2]
   \draw (0,0) rectangle +(1,1);
\end{tikzpicture}  \\ 
3 & $p_{_2}$ & \begin{tikzpicture}[scale=0.2]
   \draw (0,0) rectangle +(1,1);
   \draw (1,0) rectangle +(1,1);
\end{tikzpicture}  \\
4 & $p_{_1}^{_2}$ & \begin{tikzpicture}[scale=0.2]
   \draw (0,0) rectangle +(1,1);
   \draw (0,1) rectangle +(1,1);
\end{tikzpicture} \\
5 & $p_{_3}$ & \begin{tikzpicture}[scale=0.2]
   \draw (0,0) rectangle +(1,1);
   \draw (1,0) rectangle +(1,1);
   \draw (2,0) rectangle +(1,1);
\end{tikzpicture} \\
6 & $p_{_1}p_{_2}$ & \begin{tikzpicture}[scale=0.2]
   \draw (0,0) rectangle +(1,1);
   \draw (1,0) rectangle +(1,1);
   \draw (0,1) rectangle +(1,1);
\end{tikzpicture}  \\
7 & $p_{_4}$ & \begin{tikzpicture}[scale=0.2]
   \draw (0,0) rectangle +(1,1);
   \draw (1,0) rectangle +(1,1);
   \draw (2,0) rectangle +(1,1);
   \draw (3,0) rectangle +(1,1);
\end{tikzpicture}  \\
8 & $p_{_1}^{_3}$ & \begin{tikzpicture}[scale=0.2]
   \draw (0,0) rectangle +(1,1);
   \draw (0,1) rectangle +(1,1);
   \draw (0,2) rectangle +(1,1);
\end{tikzpicture} \\
9 & $p_{_2}^{_2}$ & \begin{tikzpicture}[scale=0.2]
   \draw (0,0) rectangle +(1,1);
   \draw (0,1) rectangle +(1,1);
   \draw (1,0) rectangle +(1,1);
   \draw (1,1) rectangle +(1,1);
\end{tikzpicture}  \\
10 & $p_{_1}p_{_3}$ & \begin{tikzpicture}[scale=0.2]
   \draw (0,0) rectangle +(1,1);
   \draw (1,0) rectangle +(1,1);
   \draw (2,0) rectangle +(1,1);
   \draw (0,1) rectangle +(1,1);
\end{tikzpicture}  \\
\hline
\end{tabular}
\label{}
\end{table}

\begin{remark}
$\mathcal{F}(2^k)$ will be a vertical $k \times 1$ figure.
$\mathcal{F}(p_k)$ will be a horizontal $1 \times k$ figure, where $p_k$ is the $k$th prime.
\end{remark}

\

Let $f, g$ be two Ferrers diagrams. We denote by $f \subseteq g$ that $f$ is a subfigure of $g$. More precisely, we will take $f \subseteq g$ to mean that figure $g$ covers figure $f$ without any protruding squares upon aligning their bottom left corners. Note that the relation $\subseteq$ is a partial ordering on $\mathcal{F}(\mathbf{N})$, with equality occurring precisely when $f$ and $g$ are identical figures. Equality of figures $f=g$ means that $f$ and $g$ have the same row lengths with the same multiplicities.

Since $\mathcal{F}$ is a bijective mapping from the naturals to the set of Ferrers diagrams, we have that $\mathcal{F}^{-1}$ is a well-defined mapping which sends a Ferrers diagram to its corresponding integer.

\section{Bounding $\pi(x)$}

\begin{lemma}\label{subfigures1}
For figures $f, g \in \mathcal{F}(\mathbf{N})$ suppose $f \subseteq g$. Then $\mathcal{F}^{-1}(f) \leq \mathcal{F}^{-1}(g)$ as natural numbers.
\end{lemma}

\begin{proof}
If $f=g$ we of course get $\mathcal{F}^{-1}(f) = \mathcal{F}^{-1}(g)$. We now assume $f \neq g$ and so $f \subset g$. In other words, $f$ is a proper subfigure of $g$.
\
We index the row lengths of $f$ as $\{f_1, \cdots f_k\}$ and the row lengths of $g$ as $\{g_1, \cdots, g_s\}$. Since $f \subseteq g$, we must have both $k \leq s$ and $f_i \leq g_i$ for all $i \leq k$ (otherwise there is a protruding square upon aligning $g$ with $f$). Since we assumed $f \neq g$, we also conclude $k < s$ or there exists some $j$ such that $f_j < g_j$.

If $k < s$, then $\mathcal{F}^{-1}(g)$ has strictly more prime divisors than $\mathcal{F}^{-1}(f)$ which are of the same or greater size since $f_i \leq g_i$ for all $i \leq s$. Therefore, $\mathcal{F}^{-1}(f) < \mathcal{F}^{-1}(g)$.

Now assume there exists some $j$ such that $f_j < g_j$. We know $\mathcal{F}^{-1}(g)$ has at least as many prime divisors as $\mathcal{F}^{-1}(f)$ which are of the same size or greater. By assumption, some $g_j$ is strictly greater than its counterpart $f_j$. Therefore, $\mathcal{F}^{-1}(f) < \mathcal{F}^{-1}(g)$.

\end{proof}

\begin{lemma}\label{subfigures2}
Let $f \in \mathcal{F}(\mathbf{N})$ be a rectangular figure with height $i$ and width $j$. The number of distinct subfigures of $f$ is ${i + j \choose j}$.
\end{lemma}

\begin{proof}
We count the number of distinct subfigures of an $i \times j$ rectangle by viewing them as combinations of $j+1$ values $\{0, 1, \cdots, j\}$.  This set of values corresponds to the possible lengths of the rows, which are distributed with repetition throughout $i$ different row slots.
The formula for combinations with repetition given $a$ values and $b$ slots is ${a + b - 1 \choose b}$. Of course we have $a = i+1$ and $b = j$.
\end{proof}

\begin{theorem}\label{euclid}
Let $x$ be a positive real number. The following bound is satisfied: $$\pi(x) \geq \frac{\lfloor \lg x \rfloor}{\lg (\lfloor \lg x \rfloor + 1)}.$$
\end{theorem}
\begin{proof}
We let $\lfloor x \rfloor$ denote the integer part of $x$. Let $M = \{1, 2, \cdots,  \lfloor x \rfloor \}$ denote the set of the first $\lfloor x \rfloor$ integers. The number of primes in $M$ is given by $\pi(x)$. Consider the set $\mathcal{F}(M)$ of figures corresponding to each of the naturals in $M$. 

We say a figure's \textit{height} is its total number of rows. The maximum height attained by any figure in $M$ is then $\lfloor \lg x \rfloor$. Letting $h := \lfloor \lg x \rfloor$, we get that the maximum height over $M$ is attained by the figure $\mathcal{F}(2^h)$. This follows from Lemma~\ref{subfigures1} since with respect to the partial ordering $\left(\subseteq, \ \mathcal{F}(\mathbf{N})\right)$ the figure $\mathcal{F}(2^h)$ is the least figure of height $\geq h$. 

We say a figure's \textit{width} is the length of its longest row. Then the maximum width attained by any figure in $M$ is precisely $\pi(x)$. Letting $w := \pi(x)$, we get that the maximum width is attained by the figure $\mathcal{F}(p_w)$, which is similarly a consequence of Lemma~\ref{subfigures1}.

Therefore, with respect to the partial ordering $\left(\subseteq, \ \mathcal{F}(\mathbf{N})\right)$ all of the figures in $M$ are bounded above by a rectangular figure with dimensions $h \times w$. 

Using Lemma~\ref{subfigures2} and the fact that the sets $M$ and $\mathcal{F}(M)$ have the same number of elements we are able to write

\begin{align*}
2^h \leq |M| = |\mathcal{F}(M)| & \leq {h + w \choose w} \\
& = \frac{(h+w)!}{h!w!} \\
& = \frac{1}{w!}(h+1)(h+2) \cdots (h+w) \\
& = (h+1)(\frac{h+2}{2}) \cdots (\frac{h+w}{w}) \\
& = (h+1)(\frac{h}{2} + 1) \cdots (\frac{h}{w} + 1) \\
& \leq (h+1)^w.
\end{align*}

Recalling that $h = \lfloor \lg x \rfloor$ and $w = \pi(x)$ and rearranging the inequality $2^h \leq (h+1)^w$ yields 

\begin{equation}\label{estimate}
\pi(x) \geq \frac{\lfloor \lg x \rfloor}{\lg (\lfloor \lg x \rfloor + 1)}.  
\end{equation}
\end{proof}

\printbibliography

\end{document}